\documentclass[11pt]{article}
\usepackage[a4paper,margin=2cm]{geometry}
\usepackage[english]{babel}
\usepackage{latexsym,amsmath,enumerate,graphics,enumerate,amsthm,tikz,hyperref,float}
\usepackage[affil-it]{authblk}
\usepackage{enumerate,amsthm,dsfont,pstricks}
\usepackage{latexsym,amsmath,amssymb,amscd,wrapfig,graphicx}

\newtheorem{theorem}{Theorem}[section]
\newtheorem{lemma}[theorem]{Lemma}
\newtheorem{proposition}[theorem]{Proposition}

\theoremstyle{definition}

\theoremstyle{remark}
\newtheorem{remark}[theorem]{Remark}

\let\phi=\varphi
\def\epsilon{\varepsilon}

\def\0{\mathbf{0}}

\newcommand{\comment}[1]{}

\numberwithin{equation}{section}

\let\epsilon=\varepsilon

\makeatletter
\def\@maketitle{%
  \newpage
  \null
  \vskip 2em%
  \begin{center}%
  \let \footnote \thanks
    {\Large\bfseries \@title \par}%
    \vskip 1.5em%
    {\normalsize
      \lineskip .5em%
      \begin{tabular}[t]{c}%
        \@author
      \end{tabular}\par}%
    \vskip 1em%
    {\normalsize \@date}%
  \end{center}%
  \par
  \vskip 1.5em}
\makeatother

\begin{document}

\title{\sc \huge Monotone dynamical systems with dense periodic points}

\author{Bas Lemmens%
\thanks{Email: \texttt{B.Lemmens@kent.ac.uk}}}
\affil{School of Mathematics, Statistics \& Actuarial Science,
University of Kent, Canterbury, CT2 7NX, United Kingdom}

\author{Onno van Gaans%
\thanks{Email: \texttt{vangaans@math.leidenuniv.nl}}}
\affil{Mathematical Institute, Leiden University, P.O.Box 9512, 2300 RA Leiden, The Netherlands}

\author{Hent van Imhoff%
\thanks{Email: \texttt{hvanimhoff@gmail.com}}}
\affil{Mathematical Institute, Leiden University, P.O.Box 9512, 2300 RA Leiden, The Netherlands}

\maketitle
\date{}
%\date{\today}

\begin{abstract}
In this paper we prove a recent conjecture by M. Hirsch, which says that if $(f,\Omega)$ is a discrete time monotone dynamical system, with $f\colon \Omega\to\Omega$  a homeomorphism  on an open connected subset of a finite dimensional vector space, and the periodic points of $f$  are dense in $\Omega$, then  $f$ is periodic.
\end{abstract}



%\tableofcontents

\section{Introduction}
The theory of monotone dynamical systems concerns the behaviour of dynamical systems on subsets of real vector spaces that preserve a partial ordering induced by a cone in the vector space. Pioneering studies by M.\,Hirsch \cite{H1,H2,H3} and numerous subsequent works, see \cite{DH,H3, HS, LNBook,PT,Sm} and the references therein, showed that under  suitable additional conditions the generic behaviour of such dynamical systems cannot be very complex. Common additional conditions include smoothness conditions on the system and various strong forms of monotonicity. 

In this paper we will consider discrete time dynamical systems $(f,\Omega)$, where $\Omega$ is an open connected subset of a finite dimensional real vector space $V$ and $f\colon \Omega\to \Omega$ is a homeomorphism which is monotone with respect to a solid closed cone $C\subseteq V$: so, $f(x)\leq_C f(y)$ whenever $x\leq_C y$ and $x,y\in \Omega$.  For such discrete dynamical systems the complexity of the generic behaviour is not well understood. Recently, however, M.\,Hirsch \cite{H4} showed that if the cone $C$ is polyhedral, then the system cannot display chaotic behaviour in the following sense. He showed that if such  a monotone dynamical system $(f,\Omega)$ has a dense set of periodic points in $\Omega$, then $f$ is periodic, that is to say, there exists an integer $p\geq 1$ such that $f^p(x)= x$ for all $x\in \Omega$. Furthermore, he conjectured that this result holds for general solid closed cones in finite dimensional vector spaces.  
The main purpose of this paper is to confirm this conjecture.  

As in \cite{H4} we will also use the following theorem, which is a direct consequence of  Montgomery \cite{Mo}. 
\begin{theorem}[Montgomery]\label{Mont} If $f\colon \Omega \to \Omega$ is a homeomorphism of an open connected subset $\Omega$ of a finite dimensional  real vector space $V$ and each $x\in \Omega$ is a periodic point of $f$, then $f$ is periodic.  
\end{theorem}

\section{Preliminaries}

We begin by recalling some basic terminology. Throughout the paper $V$ will be a  finite dimensional real vector space equipped with the usual norm topology.  A {\em cone} $C\subseteq V$ is a convex set such that $\lambda C\subseteq C$ for all $\lambda \geq 0$ and $C\cap -C =\{0\}$. The cone is said to be {\em solid} if the interior of $C$, denoted $C^\circ$, is non-empty. We will only be working with solid closed cones. Given a solid closed cone $C$, the {\em dual cone} of $C$ is given by $C^* =\{\phi\in V^*\colon \phi(x)\geq 0 \mbox{ for all } x\in C\}$, which is also solid and closed. 
Let us fix $u\in C^\circ$. Then the {\em state space} of $C$ is defined by 
\[
S_C =\{\phi \in C^*\colon \phi(u) =1\},
\]
which is a compact convex set that spans $V^*$. We write $\partial S_C$ to denote the boundary of $S_C$ with respect to its affine span.

The  cone $C$ induces a partial ordering $\leq_C $ on $V$ by $x\leq_C y$ if $y-x\in C$. We shall write $x<_Cy$ if $x\leq_C y$ and $x\neq y$, and write $x\ll_C y$ if $y-x \in C^\circ$. 
Given $x,z\in V$ we define the  {\em order-interval} $[x,z]_C :=\{y\in V\colon x\leq_C y\leq_C z\}$. Note that if $x\ll_C z$, then $[x,z]_C$ is a compact, convex set with non-empty interior, denoted $[x,z]_C^\circ$, and $[x,z]_C^\circ = \{y\in V\colon x\ll_C y\ll_C z\}$.  
A map $f\colon \Omega\to\Omega$, where $\Omega\subseteq V$, is said to be {\em monotone} (with respect to a cone $C$), if $x\leq_C y$ and $x,y\in\Omega$ implies $f(x)\leq_C f(y)$.  

We also use the following basic dynamical systems terminology. We denote a discrete time dynamical system by a pair $(f,\Omega)$, so $f\colon \Omega\to\Omega$. We say that $x\in\Omega$ is a {\em periodic point} in $(f,\Omega)$ if there exists an integer $p\geq 1$ such that $f^p(x) =x$. The least such $p\geq 1$ is called the {\em period} of $x$. A periodic point with period 1 is called a {\em fixed point}. 
The set of all periodic points of $(f,\Omega)$ is denoted $\mathrm{Per}(f)$, and the set of all fixed points of $f$ is denoted $\mathrm{Fix}(f)$.  A periodic point $x\in \Omega$ of $f$ with period $p$  is said to be {\em stable} if  there exists a 
neighbourhood $U\subseteq \Omega$ of $x$ such that $f^p(U)\subseteq U$. 

\begin{lemma}\label{stable}
If $(f,\Omega)$ be a monotone dynamical system, where $f\colon \Omega\to\Omega$ is a homeomorphism and $\mathrm{Per}(f)$ is dense in $\Omega$, then 
\begin{enumerate}[(i)] 
\item for each $x\in \Omega$ there exists $y,z\in\mathrm{Per}(f)$ such that $y \ll_C x\ll_C z$. 
\item each periodic point of $f$ is stable. 
\end{enumerate}
\end{lemma}
\begin{proof}
The first assertion follows directly from that fact that $(x+C^\circ)\cap \Omega$ and $(x-C^\circ)\cap \Omega$ are non-empty open sets, and $\mathrm{Per}(f)$ is dense in $\Omega$. 

To prove the second assertion let $x$ be a periodic point of $f$ with period $p$. Let  $y,z\in\mathrm{Per}(f)$ such that $y \ll_C x\ll_C z$. Suppose that $y$ has period $q$, and $z$ has period $r$. Let $s$ be the least common multiple of $q$ and $r$. Then 
\[
U=\bigcap_{k=0}^ {s-1} [f^{kp}(y), f^{kp}(z)]_C^\circ
\]
is a neighbourhood of $x$ such that $f^p(U) \subseteq U$. 
\end{proof}
 
The next proposition is a consequence of \cite[Proposition 6]{H4}. As the proof in \cite{H4} uses advanced results from algebraic topology, we include a more elementary proof for the reader's convenience. 
\begin{proposition}\label{dense}
Let $(f,\Omega)$ be a monotone dynamical system, where $f\colon \Omega\to\Omega$ is a homeomorphism and suppose that $\mathrm{Per}(f)$ is dense in $\Omega$. 
If $x\in\Omega$, then $\mathrm{Per}(f)$ is dense in $(x+\partial C)\cap \Omega$ and $(x-\partial C)\cap \Omega$.
\end{proposition}
\begin{proof}
We will only give the proof for  $(x+\partial C)\cap \Omega$ and leave the other case, which can be proved in an analogous way, to the reader. 
First note that by considering the monotone dynamical system $(g,\Omega -x)$, where $g(v) = f(v+x)-x$, we may as well assume that $x=0$. Now let $v\in\partial C\cap \Omega$ and $S\subseteq \Omega$ be a neighbourhood of $v$. 
Then, given $u\in C^\circ$, there exists an $\epsilon >0$ such that $[v-\epsilon u,v+\epsilon u]_C^\circ\subseteq S$. Now $(v+C^\circ)\cap [v-\epsilon u,v+\epsilon u]_C^\circ$ is an open subset of $\Omega$, and hence contains a periodic point of $f$, say $z$. Likewise, $(v-C^\circ)\cap [v-\epsilon u,v+\epsilon u]_C^\circ$ contains a periodic point of $f$, say $y$. So, $v-\epsilon u\ll_cy\ll_C v\ll_C z\ll_C v+\epsilon u$, and $[y,z]_C\subset [v-\epsilon u,v+\epsilon u]_C^\circ$. 

Let $r$ be the least common multiple of the periods of $y$ and $z$, and write $g=f^r$. So, $y,z\in \mathrm{Fix}(g)$. Now let $\mathcal{M}$ be the collection of $M\subseteq \mathrm{Fix}(g)$ such that $M$ is totally $\leq_C$-ordered with $\min M =y$ and $\max M =z$, and order $\mathcal{M}$ by inclusion. Then  each chain $(M_\alpha)$ in $(\mathcal{M},\subseteq)$ has an upper bound, namely $\cup_\alpha M_\alpha$. Indeed, if $a,b\in \cup_\alpha M_\alpha$, then there exists an $\alpha$ such that $a,b\in M_\alpha$, and hence either $a\leq_Cb$ or $b\leq_C a$, as $M_\alpha$ is totally $\leq_C$-ordered. Thus, by Zorn's Lemma $(\mathcal{M},\subseteq)$ has a maximal element, say $M$. 

We claim that $M$ is a connected subset of $\Omega$. To show this we argue by contradiction. So, suppose that there exist $U,W\subseteq M$ non-empty and relatively open such that $U\cap W=\emptyset$ and $M= U\cup W$. We may as well assume that $y\in U$. Note that both $U$ and $W$ are totally $\leq_C$-ordered. Thus, $\{x_w\}$, $w\in (W,\preceq)$, forms a net, where $w_1\preceq w_2$ if $w_2\leq_C w_1$ and $x_w =w$ for all $w\in W$. Now for each $\phi\in S_C$ we have that $\{\phi(x_w)\}$ is a decreasing net that is bounded below by $\phi(y)$, and hence it converges. As $S_C$ spans $V^*$, we conclude that $\{x_w\}$ converges to say $w^*\in [y,z]_C$. As $C$ is closed, we get that $w^*\leq_C w$ for all $w\in W$.

Now let $U_0=\{u\in U\colon u\leq_C w^*\}$. Then $y\in U_0$ and $U_0$ is totally $\leq_C$-ordered. Thus, $\{x_u\}$, $u\in (U_0,\preceq)$ forms a net, where $u_1\preceq u_2$ if $u_1\leq_C u_2$ and $x_u = u$ for all $u\in U_0$. As before, $\{x_u\}$, $u\in U_0$, converges to say $u^*\in [y,z]_C$, and $u\leq_C u^*$ for all $u\in U_0$. 

Note that as $\mathrm{Fix}(g)$ is closed, $w^*$ and $u^*$ are fixed points of $g$. To derive a contradiction that proves the claim we distinguish two cases: $u^*\neq w^*$ and $u^*=w^*$. If $u^*\neq  w^*$, then $u^*<_C w^*$ and $u^*$ and $w^*$ are stable fixed points of $g$ by Lemma \ref{stable}(ii). It now follows from \cite[Proposition 1]{DH} that there exists $\eta\in \mathrm{Fix}(g)$ with $u^*<_C\eta <_C w^*$. Note that $U\cup W=M$ implies that $\eta\not\in M$ and $M\cup\{\eta\}\in\mathcal{M}$, as each $w\in W$ satisfies $\eta<_C w^*\leq_C w$, each $u\in U_0$ satisfies $u\leq_C u^*<_C\eta$, and for each $u\in U\setminus U_0$ there exists a $w\in W$ with $\eta<_Cw^*\leq_Cw\leq_C u$. This, however, contradicts the maximality of $M$. On the other hand, if $u^*=w^*$, then writing $\xi = u^*=w^*$ we see that $\xi\not\in M$. Indeed, if $\xi \in M$, then either $\xi \in U$ or $\xi \in W$, which is impossible as $\xi \in \overline{U}\cap\overline{W}$, $U$ and $W$ are relatively open, and $U\cap W=\emptyset$.  As in the previous case, one can check that $M\cup\{\xi\}$ belongs to $\mathcal{M}$, which again contradicts the maximality of $M$. This shows that $M$ is connected. 

To complete the proof we consider the continuous  function $h\colon V\to\mathbb{R}$ given by 
\[
h(x)= \inf_{\phi \in S_C} \phi(x)\mbox{\quad for }x\in V.
\]  
Recall that $0\leq_C v\ll_C z$ and $y\ll_C v$. So, for each $\phi\in S_C$ we have that $0\leq \phi(v)<\phi(z)$, which implies that $h(z) >0$, since $S_C$ is compact. Moreover, there exists $\phi^*\in S_C$ such that $\phi^*(v) =0$, so that $\phi^*(y)<\phi^*(v)=0$. This implies that $h(y)<0$. Now, as $h$ is continuous and $M$ is connected, $h(M)$ is a connected subset of $\mathbb{R}$, and hence there exists $\zeta\in M$ such that $h(\zeta) =0$. It follows that $\phi(\zeta) \geq 0$  for all $\phi\in S_C$, and there exists $\phi'\in S_C$ such that $\phi'(\zeta)=0$, since $S_C$ is compact. Thus, $\zeta\in\partial C\cap[y,z]_C\subseteq \partial C\cap S$ and $f^r(\zeta) =g(\zeta)= \zeta$, which completes the proof.
\end{proof}
The proof of  the main result  relies on  a couple of geometric lemmas  concerning convexity of finite dimensional solid closed cones. For basic convexity notions we follow the terminology of \cite{Rock}.   Given $x\in \partial C$ we  write
\[
\nu(x) =\{\phi \in \partial S_C\colon \phi(x) =0\}.
\]
Note that $\nu(x) = \nu(\lambda x)$ for all $\lambda>0$, and $\nu(x)$ is non-empty for each $x\in\partial C$, as each $x\in\partial C$ has a supporting functional. 

We will consider $\partial S_C$ and $\partial C$ as topological spaces with the induced norm topology from $V^*$ and $V$, respectively.

\begin{lemma}\label{open} If $U\subseteq \partial S_C$ is open, then $\{x\in\partial C\colon \nu(x) \subseteq U\}$ is open.
\end{lemma}
\begin{proof}
Suppose by way of contradiction that there exists $z\in\partial C$ with $\nu(z)\subseteq U$ and a sequence $\{z_n\}$ in $\partial C$ converging to $z$ with $\nu(z_n)\not\subseteq U$ for all $n\geq 1$. Then for each $n\geq 1$ there exists  a $\phi_n\in \nu(z_n)$ with $\phi_n\not\in U$. As $\partial S_C$ is compact, we may assume after taking a subsequence that $\{\phi_n\}$ converges to $\phi \in\partial S_C$.  Now note that 
\[
0\leq \phi(z) \leq |\phi(z) -\phi(z_n)| +|\phi(z_n)-\phi_n(z_n)|\leq \|\phi\|\|z-z_n\| +\|\phi-\phi_n\|\|z_n\|
\]
for all $n\geq 1$. Since the right-hand side converges to $0$ as $n\to \infty$, we conclude that $\phi(z) =0$, and hence  $\phi\in\nu(z)\subseteq U$. This is impossible, since $\{\phi_n\}$ converges to $\phi \in U$, $\phi_n\not\in U$ for all $n\geq 1$, and $U$ is open. 
\end{proof}
\begin{remark} \label{remark} Note that given $U\subseteq \partial S_C$ open, the set $\{x\in\partial C\colon \nu(x)\subseteq U\}$ may be empty.
If, however, $\phi\in \partial S_C$ is an exposed point of $S_C$, then by definition there exists $y\in\partial C$ such that  $\phi(y) =0$ and $\psi(y)>0$ for all $\psi\neq \phi$ in $\partial S_C$. 
In that case $\nu(\lambda y) =\{\phi\}$ for all $\lambda >0$. Thus, for any neighbourhood $U$ of an exposed point $\phi$ in $\partial S_C$ we know that $\{x\in\partial C\colon \nu(x)\subseteq U\}\cap W$ is non-empty and open, for all non-empty neighbourhoods $W$ of $0$ in $V$.
\end{remark}

By Straszewicz's Theorem \cite[Theorem 18.6]{Rock} the exposed points of $S_C$ are dense in the extreme points of $S_C$.  As $S_C$ is the convex hull of its extreme points,   $S_C$ is also the convex hull of its exposed points.  Let $\psi_1,\ldots,\psi_d$ be linearly independent  exposed points of $S_C$, where $d=\dim V^*=\dim V$, and let 
\begin{equation}\label{K}
K =\left\{\sum_{i=1}^d \lambda_i\psi_i\colon \lambda_1,\ldots,\lambda_d\geq 0\right \},
\end{equation}
which is a solid closed cone in $V^*$. The dual cone of $K$ is $K^*=\{x\in V\colon \psi_i(x)\geq 0\mbox{ for all }i=1,\ldots,d\}$,
which is also closed and solid. Furthermore let 
\begin{equation}\label{psi}
\psi= \frac{1}{d}\sum_{i=1}^d\psi_i.
\end{equation}
Then $\psi$ is  a strictly positive functional for $K^*$, that is to say, $\psi(x)>0$ for all $x\in K^*\setminus\{0\}$.  
\begin{lemma}\label{perturb} For $i=1,\ldots,d$ there exist  neighbourhoods $U_i$ of $\psi_i$ in $\partial S_C$ such that if $\phi_i\in U_i$ for $i=1,\ldots d$, then
\begin{enumerate}[(i)] 
\item $\phi_1,\ldots,\phi_d$ are linearly independent. 
\item $\psi$ is a strictly positive functional for the solid closed cone \[K'=\{x\in V\colon \phi_i(x)\geq 0\mbox{ for all } i=1,\ldots,d\}.\]
\end{enumerate} 
\end{lemma}
\begin{proof}
The first assertion follows directly from the fact that the set of invertible linear maps $L\colon \mathbb{R}^d\to V^*$ is open by considering the invertible linear map 
$A\colon x\mapsto \sum_{i=1}^d x_i\psi_i$ for $x\in\mathbb{R}^d$. 

The second assertion follows from the fact that the map $L\mapsto L^{-1}$ is continuous on the set of invertible linear maps from $\mathbb{R}^d$ onto $V^*$.
Indeed, consider $A$ as above. Now if $\phi_1,\ldots,\phi_d$ are linearly independent, then the linear map $B\colon x\mapsto \sum_{i=1}^d x_i\phi_i$ is invertible. 
Thus, 
\[
\|B^{-1}\psi -(1/d,\ldots,1/d)\| = \|B^{-1}\psi -A^{-1}\psi\|\leq \|B^{-1}-A^{-1}\|\|\psi\|,
\] 
implies that $(B^{-1}\psi)_i>0$ for all $i=1,\ldots d$, if $B^{-1}$ is sufficiently close to $A^{-1}$.  As $\psi = B(B^{-1}\psi) = \sum_{i=1}^d (B^{-1}\psi)_i\phi_i$, we conclude that 
$\psi(x) =  \sum_{i=1}^d (B^{-1}\psi)_i\phi_i(x) >0$ for all $x\in K'\setminus\{0\}$, since $\phi_j(x)>0$ for some $j$ and $(B^{-1}\psi)_i>0$ for all $i$. 
\end{proof}

\section{Proof of periodicity: Hirsch's conjecture}
\begin{theorem}
If $(f,\Omega)$ is a monotone dynamical system, where $f\colon \Omega\to\Omega$ is a homeomorphism on an open connected subset $\Omega$ of a finite dimensional real vector space $V$ and  $\mathrm{Per}(f)$ is dense in $\Omega$, then $f$ is periodic.
\end{theorem}
\begin{proof}
By Montgomery's Theorem \ref{Mont} it suffices to show that each $x\in\Omega$ is periodic. So let $x\in\Omega$. By considering the dynamical system $(g,\Omega -x)$, where $g(v) = f(v+x)-x$ for $v\in\Omega-x$, we may as well assume that $x=0$. 
From  Proposition \ref{dense} we know that the periodic points of $f$ are dense in $\partial C\cap \Omega$ and $-\partial C\cap \Omega$. 

Now as above choose $\psi_1,\ldots,\psi_d$ linearly independent  exposed points of $S_C$, and let $\psi = \frac{1}{d}\sum_{i=1}^ d \psi_i$. By Lemma \ref{perturb} there exist open neighbourhoods $U_i$ of $\psi_i$ in $\partial S_C$ for $i=1,\ldots,d$ such that if we take $\phi_i\in U_i$ for $i=1,\ldots,d$, then $\phi_1,\ldots,\phi_d$ are linearly independent, and $\psi$ is a strictly positive functional for the solid closed cone $K'=\{x\in V\colon \phi_i(x)\geq 0\mbox{ for all } i=1,\ldots,d\}$. 

Lemma \ref{open} implies that for each $i$ we have that $W_i=\{x\in\partial C\colon \nu(x)\subseteq U_i\}$ is open in $\partial C$. It follows that  $W_i\cap \Omega$ and  $-W_i\cap \Omega$ are non-empty and open in $\partial C\cap \Omega$ and $-\partial C\cap\Omega$, respectively, as $\psi_i$ is an exposed point of $\partial S_C$, see Remark \ref{remark}.  As  the periodic points of $f$ are dense in $-\partial C\cap \Omega$, there exists a periodic point $x_i\in -W_i\cap \Omega$ of $f$ with period, say $p_i$, for $i=1,\ldots, d$.  Let $\rho_i\in U_i$ be such that $\rho_i\in \nu(x_i)$.  So, $\rho_1,\ldots,\rho_d$ are linearly independent and $\rho_i(x_i)=0$ for $i=1,\ldots, d$. 

Now consider the set $S_1=\{y\in \Omega \colon x_i\leq_C y \mbox{ for all }i=1,\ldots,d\}$. Then $0\in S_1$ and  for each $y\in S_1$ we have that $\rho_i(y)\geq \rho_i(x_i)=0$ for all $i=1,\ldots, d$. So 
$y$ belongs to the solid closed cone $K_1'=\{v\in V\colon \rho_i(v)\geq 0\mbox{ for all }i=1,\ldots,d\}$. Moreover, $\psi$ is a strictly positive functional for $K'_1$. 

Likewise, there exist periodic points $z_1,\ldots,z_d$ with $z_i\in W_i\cap\Omega$ for  $i=1,\ldots,d$. Let $q_i$ be the period of $z_i$, and take $\sigma_i\in \nu(z_i)\subseteq U_i$ for all $i$. 
Then $\sigma_1,\ldots,\sigma_d$ are linearly independent and $\sigma_i(z_i)=0$ for all $i$. Now consider $S_2= \{y\in \Omega \colon y\leq_C z_i \mbox{ for all }i=1,\ldots,d\}$. 
So, $0\in S_2$ and for each $y\in S_2$ we have that $\sigma_i(y) \leq \sigma_i(z_i)=0$ for all $i$, and hence $y$ belongs to $-K_2'$, where $K_2'=\{v\in V\colon \sigma_i(v)\geq 0\mbox{ for all }i=1,\ldots,d\}$ is a solid closed cone. 
Again $\psi$ is a strictly positive linear functional for $K'_2$. 

Thus, $S_1\cap S_2\subseteq K_1'\cap (-K'_2)$ and $0\in S_1\cap S_2$. In fact, we have that $S_1\cap S_2= \{0\}$. Indeed, if $y\in S_1\cap S_2$, then $\psi(y) \geq 0$, as $y\in  K_1'$. But also $\psi(y)\leq 0$, as $y\in  -K_2'$. Thus $\psi(y) =0$, which implies that $y=0$, since $\psi$ is a strictly positive linear functional for $K_1'$. 

To complete the proof note that if we let $r$ be the least common multiple of $p_1,\ldots,p_d, q_1,\ldots,q_d$, then $f^r(0) \in S_1\cap S_2$, as $f$ is monotone. Thus, $f^r(0) =0$ and we are done. 
\end{proof}

\end{document}